\documentclass[12pt]{article}
\usepackage{mathrsfs}
\usepackage{amsthm}
\usepackage{amssymb}
\usepackage{latexsym}
\usepackage{amsmath,amsfonts}
\usepackage{mathrsfs}
\usepackage{cases}
\usepackage{latexsym,bm}
\usepackage{indentfirst}
\usepackage{xcolor}
\usepackage{ifpdf}
\usepackage{graphicx}
\usepackage{epstopdf}
\usepackage{epsfig}
\usepackage{psfrag}
\usepackage{enumitem}
\usepackage{epstopdf}
\usepackage{verbatim}
\usepackage{color}
\usepackage{subfigure}
\usepackage{caption}

\usepackage[
pdfauthor={Lu},
pdftitle={matching covered},
pdfstartview=XYZ,
bookmarks=true,
colorlinks=true,
linkcolor=blue,
urlcolor=blue,
citecolor=blue,
bookmarks=true,
linktocpage=true,
hyperindex=true
]{hyperref}
 
\usepackage[textwidth=16cm, textheight=22cm, includehead, includefoot]{geometry}

\title{Adjacent vertices of small degree in minimal matching covered graphs
\footnote{The research is supported by NSFC (Grant Nos. 12671409, 12271229 and 12271235), Natural Science Foundation of Fujian Province (No. 2026J002034), Fujian Key Laboratory of Granular Computing and Applications, Institute of Meteorological Big Data-Digital Fujian and Fujian Key Laboratory of Data Science and Statistics.
\newline E-mail addresses: xiaolinghe99@163.com (X. He), flianglu@163.com (F. Lu), zhanghp@lzu.edu.cn (H. Zhang).} }
\author{ Xiaoling He $^1$, Fuliang Lu $^2$, Heping Zhang $^1$\footnote{The corresponding author.}\\
\small {$^1$ School of Mathematics and Statistics, Lanzhou University, Lanzhou, China}\\
\small {$^2$ School of Mathematics and Statistics, Minnan Normal University, Zhangzhou, China}
}

\date{}

\newtheorem{lem}{Lemma}[section]
\newtheorem{thm}[lem]{Theorem}
\newtheorem{cor}[lem]{Corollary}

\newtheorem{pro}[lem]{Proposition}
\newtheorem{conj}[lem]{Conjecture}

\newtheorem{remark}[lem]{Remark}

\newtheorem*{claim}{Claim}

\makeatletter
\renewenvironment{claim}[1][]{%
  \par\addvspace{\topsep}
  \noindent{\textbf{Claim #1.}}\ignorespaces%
  \itshape
}{\par\addvspace{\topsep}}
\makeatother

\begin{document}
\bibliographystyle{plain}
\newcommand{\udots}{\mathinner{\mskip1mu\raise1pt\vbox{\kern7pt\hbox{.}}
\mskip2mu\raise4pt\hbox{.}\mskip2mu\raise7pt\hbox{.}\mskip1mu}}
\maketitle
\begin{abstract}
A connected  graph $G$ with at least two vertices is {\em matching covered} if each of its edges lies in a perfect matching.
A matching covered graph is {\em minimal} if the removal of any edge results in a graph that is no longer matching covered.
An edge is called a {\em $k$-line} if both of its end vertices are of degree $k$.
Lov\'asz and Plummer [J. Combin. Theory, Ser. B  23 (1977) 127--138] proved that a minimal matching covered bipartite graph different from $K_2$ has minimum degree 2 and contains at least $[(|V(G)|+15)/6]$ 2-lines by ear decompositions.
He et al. [J. Graph Theory 111 (2026) 5--16] showed that the minimum degree of a minimal matching covered graph different from $K_2$ is either 2 or 3.
In this paper, we prove that every minimal matching covered graph with at least 4 vertices contains at least two nonadjacent edges, each of which is either a 2-line or a 3-line. Consequently, we show that every minimal matching covered graph with at least 4 vertices and minimum degree 3 contains at least 4 vertices of degree 3. Furthermore, the lower bounds for both the number of 3-lines and the number of cubic vertices are sharp.
\\

\end{abstract}
\vskip 0.2in \baselineskip 0.1in
 
\section{Introduction}

Graphs considered in this paper may have multiple edges, but no loops. We follow \cite{BM08} for undefined notation and terminology.
Let $G$ be a graph with the vertex set $V(G)$ and the edge set $E(G)$.
For a vertex $u\in V(G)$, the {\em degree} of $u$ in $G$,
denoted by $d_G(u)$ or simply $d(u)$, is the number of edges of $G$ incident with $u$.
Specifically, a vertex $u$ is called a {\em cubic vertex} of $G$ if $d_G(u)=3$.
We denote by $\delta(G)$ and $\Delta(G)$ the {\em minimum degree} and the {\em maximum degree} of $G$, respectively.

An edge $e$ of a graph $G$  is {\em allowed} if there exists some perfect matching of $G$ containing $e$.
A connected nontrivial graph $G$ is {\em matching covered} if each of its edges is allowed.
We say that an edge $e$ in a matching covered graph $G$ is {\em removable} if $G-e$ is matching covered. A pair $\{e,f\}$ of edges of a matching covered graph $G$ is a {\em removable doubleton} if $G-e-f$ is matching covered, but neither $G-e$ nor $G-f$ is. Removable edges and removable doubletons are called {\em removable classes}.
We say that a matching covered graph $G$ is {\em minimal} if $G-e$ is not a matching covered graph for any edge $e$ in $G$, equivalently, $G$ contains no removable edges.

An edge of a graph is called a {\em $k$-line} if both of its end vertices are of degree $k$.
For a minimal matching covered bipartite graph $G$ different from $K_2$, Lov\'{a}sz and Plummer \cite{LP77} proved that $\delta(G)=2$ and that every ear in any ear decomposition of $G$ contains at least one 2-line \footnote{Lov\'{a}sz and Plummer used the terminology ``minimal elementary bipartite graph'' in \cite{LP77}.}.
Consequently, they have the following.
\begin{thm}[\cite{LP77}]\label{thm:2-lines}
    Every minimal matching covered bipartite graph $G$ different from $K_2$ contains at least $[(|V(G)| + 15)/6]$ 2-lines.
\end{thm}
Moreover, Lov\'{a}sz and Plummer \cite{LP77} showed that the lower bound of Theorem \ref{thm:2-lines} is sharp for all $|V(G)| \ge 8$.
Generally, He et al. \cite{HLX2025} proved that the minimum degree of a minimal matching covered graph different from $K_2$ is either 2 or 3.  We focus on adjacent vertices of minimum degree in minimal matching covered graphs in this paper. The following are our main results.
\begin{thm}\label{thm:main-thm}
    Let $G$ be a minimal matching covered graph with at least four vertices and $\delta(G)\ge3$. Then $G$ contains at least two (nonadjacent) 3-lines \footnote{In \cite{LP77}, the terminology  ``3-line'' refers to an edge incident with two vertices of degree at least 3, but in this paper, we use this terminology to denote an edge incident with two vertices of degree exactly 3.}. Consequently, $G$ contains at least 4 cubic vertices.
\end{thm}

\begin{remark} {\rm
The lower bounds in  Theorem \ref{thm:main-thm} are sharp for graphs with at least six vertices.
Let the graph $G_{n}$ $(n\ge3)$ be obtained from two disjoint paths $u_1u_2\ldots u_n$ and $v_1v_2\ldots v_n$ by adding the edge set $\{u_1v_1,u_nv_n\}\cup \{u_iv_{i+1}, v_iu_{i+1}:i=1,2,\dots, n-1\}$ (see Figure \ref{fig:sharp-exp} when $n=5$).
Note that $\{u_j,v_j\}$ is a 2-separation of $G_{n}$  (see its definition in Section 2), for each $2\le j\le n-1$.
 Then $G_n$ can be gotten by the splicings (see its definition in Section 2) of $n-1$ bricks whose underlying simple graphs are isomorphic to $K_4$ (the complete graph with 4 vertices).
So $G_{n}$ is matching covered (by Proposition \ref{thm:MC_IS_MC}).
Moreover, it can be checked that $u_2u_3$ is not allowed in $G_{n}-u_1v_1$ and  $u_{n-2}u_{n-1}$ is not allowed in $G_{n}-u_nv_n$; and for each $1\le i\le n-1$, $v_iv_{i+1}$ is not allowed in $G_{n}-u_iu_{i+1}$ and $u_iv_{i+1}$ is not allowed in $G_{n}-v_iu_{i+1}$ (the statements also hold when $u$ and $v$ are interchanged).  It means that every edge of $G_{n}$ is not removable, that is, $G_{n}$ is a minimal matching covered graph. Moreover, for every integer $n\ge3$, the graph $G_n$ has exactly two 3-lines: $u_1v_1$ and $u_nv_n$ (which are nonadjacent), and has exactly 4 cubic vertices: $u_1$, $v_1$, $u_n$ and $v_n$.
}
    \begin{figure}[!h]
    \centering
    \includegraphics[totalheight=3cm]{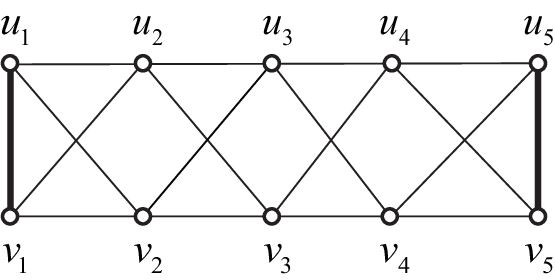}
    \caption{The graph $G_{5}$ (the bold edges are 3-lines).}
    \label{fig:sharp-exp}
\end{figure}
\end{remark}
\begin{thm}\label{thm:2-or-3-lines}
        Every minimal matching covered graph with at least four vertices  contains two nonadjacent edges such that each of them is a 2-line or 3-line.
\end{thm}
\begin{remark}\label{remark:without-2-lines}
    {\rm Not every minimal matching covered nonbipartite graph with  minimum degree 2 has a 2-line. An example is given as follows.
    Let the graph $H_n$ ($n\ge3$) be obtained by the splicing of $G_n$ at $u_2$ with $C_4^+$ at $w$, where $C_4^+$ is the graph shown in  Figure \ref{fig:couter-exp}. It can be checked by Proposition \ref{thm:MC_IS_MC} that $H_n$ is matching covered and by Lemma \ref{lem:re_also_re} that  $H_n$ is free of removable edges (as $\partial(\{a,a',a''\})$ is a tight cut of $H_n$). For every integer $n\ge3$, the graph $H_n$ is a minimal matching covered graph with only one vertex of degree 2 (see Figure \ref{fig:couter-exp} when $n=5$)}.
\begin{figure}[!h]
\centering
\begin{minipage}[b]{0.28\textwidth}
\centering
\includegraphics[totalheight=3cm]{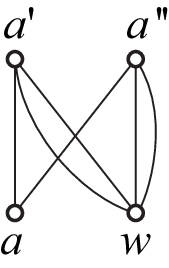}
\end{minipage}
\begin{minipage}[b]{0.68\textwidth}
\centering
\includegraphics[totalheight=4cm]{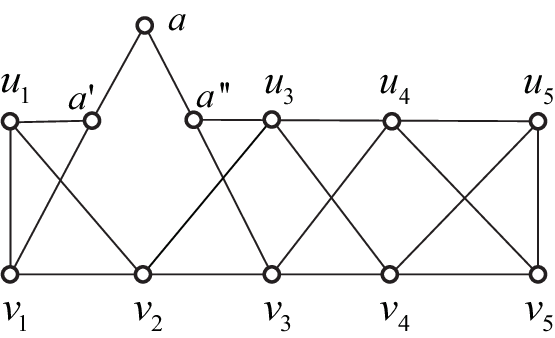}
\end{minipage}
\caption{The graphs $C_4^+$ (left), and   $H_5$ (right).}
\label{fig:couter-exp}
\end{figure}
\end{remark}

In Section 2, we present some notation and preliminary results. 
In Section 3,  we will consider the removable edges in matching covered  bipartite graphs.
In Section 4, some properties of removable edges in bicritical graphs will be presented.
In Section 5, we give proofs of Theorems \ref{thm:main-thm} and \ref{thm:2-or-3-lines}. 
 In Section 6, we present a conjecture that states that each minimal brace (see its definition in Sections 4 and 6)  with at least 6 vertices has a 3-line.

\section{Preliminaries}

We begin with some notation.
 For nonempty sets $X$ and $Y$, the notation $X \subseteq Y$ means that $X$ is a subset of $Y$, while $X \subset Y$ means that $X \subseteq Y$ with $X \neq Y$. Likewise, $H \subset G$ denotes that $H$ is a subgraph of the graph $G$ and $H \neq G$.
For a vertex $u\in V(G)$, denote by $N_G(u)$ or simply $N(u)$  the set of vertices in $G$ adjacent to $u$.
 For $\emptyset\neq X\subset V(G)$, let $N_G(X)$ denote the set of all vertices in $V(G)\setminus X$ which are adjacent to at least one vertex in $X$ and let $G[X]$ be the subgraph of $G$ induced by $X$.
For nonempty proper subsets $X,Y\subset V(G)$, by $E_G[X,Y]$ we mean the set of edges of $G$ with one end vertex in $X$ and the other end vertex in $Y$. Let $\partial_G(X) := E_G[X,\overline{X}]$ be an {\em edge cut} of $G$, where $\overline{X} = V(G) \backslash X$. If $X = \{u\}$, then we denote $\partial_G(\{u\})$, for brevity,  by $\partial_G(u)$.  (If $G$ is understood, the subscript $G$ will be omitted.)
An edge cut $\partial(X)$ is {\em trivial} if $|X| = 1$ or $|\overline{X}| = 1$.
Let $\partial(X)$ be an edge cut of $G$.
Denote by $G/(X\to x)$, or simply $G/X$, the graph obtained from $G$ by contracting $X$ to a singleton $x$ and removing any resulting loops. The graphs $G/X$ and $G/\overline{X}$ are {\em $\partial(X)$-contractions} of $G$.
 \begin{lem}[\cite{Lovasz87}]\label{lem:C-contractions-MC}
   For any nontrivial tight cut $\partial(X)$ of a matching covered graph $G$, both $\partial(X)$-contractions of $G$ are matching covered.
\end{lem}
 A component with odd (even) number of vertices is called an {\em odd (even) component}. A component is {\em trivial} if it contains exactly one vertex.
We denote by $o(G)$ the number of odd components of the graph $G$.
A vertex subset $B$ of a graph $G$ that has a perfect matching is a {\em barrier} if $o(G-B)=|B|$. A barrier is {\em trivial} if it contains at most one vertex.
 Tutte  proved the following classical theorem in 1947.
\begin{thm}[\cite{Tutte47}]\label{thm:Tutte}
    A graph $G$ has a perfect matching if and only if $o(G-S)\le|S|$, for every $S\subseteq V(G)$.
\end{thm}
Using Theorem \ref{thm:Tutte}, we have the following properties about matching covered graphs.
\begin{cor}[\cite{Lovasz87}]\label{cor:M-C-without-even-components}
  Let $G$ be a matching covered graph. Every nonempty barrier $B$ is independent, and $G-B$ has no even components.
\end{cor}

A graph $G$ with four or more vertices is {\em bicritical} if for any two distinct vertices $u$ and $v$ in $G$, $G-\{u, v\}$ has a perfect matching. Obviously, every bicritical graph is matching covered.
\begin{pro}[\cite{LP86}]\label{pro:mc_is_Bi}
    A matching covered graph $G$ with at least 4 vertices is bicritical if and only if every barrier of $G$ is trivial.
\end{pro}
A vertex subset $S$ of a matching covered graph $G$ is a {\em 2-separation} if $|S|=2$, $G-S$ is disconnected and each of the components of $G-S$ is even.
 The following corollary can be gotten directly by Theorem \ref{thm:Tutte} and Proposition \ref{pro:mc_is_Bi}.
\begin{cor}\label{cor:2-vtx-cut-of-MC}
    Let $G$ be a matching covered graph different from $K_2$ and let $u,v\in V(G)$. If $G-\{u,v\}$ is disconnected, then $\{u,v\}$ is either a barrier of $G$ or a 2-separation of $G$;  moreover, if $G$ is bicritical, then $\{u,v\}$ is a 2-separation of $G$.
\end{cor}
 A graph $G$ is {\em factor-critical} if for any $v\in V(G)$, the graph $G-v$ has a perfect matching. Obviously, every bipartite graph with at least 3 vertices is not  factor-critical.  The following lemma is easy to verify by the definition (e.g., see Exercise 3.3.18 (b) of \cite{LP86} or Theorem 3.2 of \cite{Lucchesi2024}).
\begin{lem}\label{lem:max-barrier-critical}
    Let $B$ be a maximal barrier in a graph $G$ that has a perfect matching. Then each component of $G-B$ is a factor-critical graph.
\end{lem}
An edge cut $\partial(X)$ is {\em tight} of a matching covered graph $G$ if $|\partial(X)\cap M|=1$ for every perfect matching $M$ of $G$. Obviously, a trivial edge cut is a tight cut.
Two special types of tight cuts in matching covered graphs, called barrier cuts and 2-separation cuts, are defined as follows.

An edge cut $C$ of a matching covered graph $G$ is a {\em barrier cut} if there exists a barrier $B$ of $G$ and an odd component $Q$ of $G-B$ such that $C=\partial(V(Q))$.
Let $\{u,v\}$ be a 2-separation of a matching covered graph $G$, and let us divide the components of $G-\{u,v\}$ into two nonempty subgraphs and denote one of them by $G_1$. The cuts $\partial(V(G_1)\cup \{u\})$ and $\partial(V(G_1)\cup \{v\})$ are both {\em 2-separation cuts} of $G$ associated with $\{u,v\}$, which are nontrivial tight cuts.

Denote by $G[A, B]$ the bipartite graph with bipartition $(A,B)$. The following is a Hall-type characterization of matching covered bipartite graphs.
\begin{lem}[\cite{Lucchesi2024}]\label{lem:bipartite-M-C-Hall}
    Let $G[A,B]$ be a bipartite graph with at least 4 vertices such that $|A| = |B|$. The graph  $G[A,B]$ is matching covered if and only if $|N(S)| \ge |S| + 1$, for all subsets $S$ of $A$ such that $1 \le |S| \le |A|-1$.
\end{lem}

Let $G$ and $H$ be two vertex-disjoint graphs and let $u$ and $v$ be vertices of $G$ and $H$, respectively, such that $d_G(u)=d_H(v)$.
Moreover, let $\theta$ be a given bijection between $\partial_H(v)$ and $\partial_G(u)$.
We denote by $(G(u)\odot H(v))_\theta$ the graph obtained from the union of $G-u$ and $H-v$ by joining, for each edge $e$ in $\partial_H(v)$, the end vertex of $e$ in $H$ belonging to $V(H)-v$ to the end vertex of $\theta(e)$ in $G$ belonging to $V(G)-u$;
and refer to $(G(u)\odot H(v))_\theta$ as the graph obtained by \emph{splicing $G$ (at $u$)   with $H$ (at $v$), with respect to the bijection $\theta$}, for brevity, to $G(u)\odot H(v)$. We say that $u$ and $v$ are the {\em splicing vertices} of $G$ and $H$, respectively.
In general, the graph resulted from splicing two graphs $G$ and $H$ depends on the choice of $u$, $v$ and $\theta$. The following proposition can be gotten by the definition of matching covered graphs directly (e.g., see Theorem 2.13 in \cite{Lucchesi2024}).

\begin{pro}\label{thm:MC_IS_MC}
    The splicing of two matching covered graphs is also matching covered.
\end{pro}

\section{Removable edges in bipartite graphs}
The following lemma is easy to verify by the definition of matching covered graphs.
 \begin{lem}[\cite{CLM99}]\label{lem:re_also_re}
    Let $G$ be a matching covered graph with a nontrivial tight cut $C$. Let $G_1$ and $G_2$ be the two $C$-contractions of $G$. Then for an edge $e$ in $G$, $G-e$ is matching covered if and only if $G_1-e$ and $G_2-e$ are matching covered.
\end{lem}

\begin{lem}[\cite{CLM15}]\label{lem:nonre-bi}
    Let $G[A,B]$ be a matching covered bipartite graph, and $|E(G)|\ge2$.
    An edge $uv$ of $G$, with $u\in A$ and $v\in B$, is not removable in $G$ if and only if there exist nonempty proper subsets $A_1$ and $B_1$ of $A$ and $B$, respectively, such that:\\
    {\rm (i)} the subgraph $G[A_1\cup B_1]$  is matching covered, and\\
    {\rm (ii)} $u\in A_1$ and $v\in B\setminus B_1$, and $E[A_1,B\setminus B_1 ]=\{uv\}$.
\end{lem}

Let $G[A,B]$ be a matching covered bipartite graph with at least two edges and let $X$ be a vertex subset of $G$ such that $|X\cap A|=|X\cap B|$.
We say that $X$ is a {\em $P$-set} of $G$ if $G[X]$ has a perfect matching, and either $|E[X\cap A,\overline{X}\cap B]|=1$ or $|E[\overline{X}\cap A,X\cap B]|=1$. Obviously, $\overline{X}$ is a $P$-set if $X$ is a $P$-set. If $E[X\cap A,\overline{X}\cap B]=\{e\}$, we say that $X$ is a $P$-set {\em associated with} $e$, and $X\cap A$ and $\overline{X}\cap B$ are a pair of {\em end sets} of $e$.
By Lemma \ref{lem:nonre-bi}, a nonremovable edge must have a pair of end sets (not unique).

\begin{lem}\label{lem:nonre-forest}
    Let $G[A,B]$ be a matching covered bipartite graph with at least 4 vertices. If $\delta(G)\ge3$, then the subgraph induced by all nonremovable edges of $G$ is a forest.
    \begin{proof}
        Assume that $uv_1$ and $uv_2$ are two adjacent nonremovable edges of $G$ with a common end vertex $u$, where $u\in A$. By Lemma \ref{lem:nonre-bi}, we may assume that $X$ and $Y$ are two $P$-sets associated with $uv_1$ and $uv_2$, respectively.
        Adjust notation so that $u\in X\cap Y$.
    \begin{figure}[!h]
    \centering
    \includegraphics[totalheight=3cm]{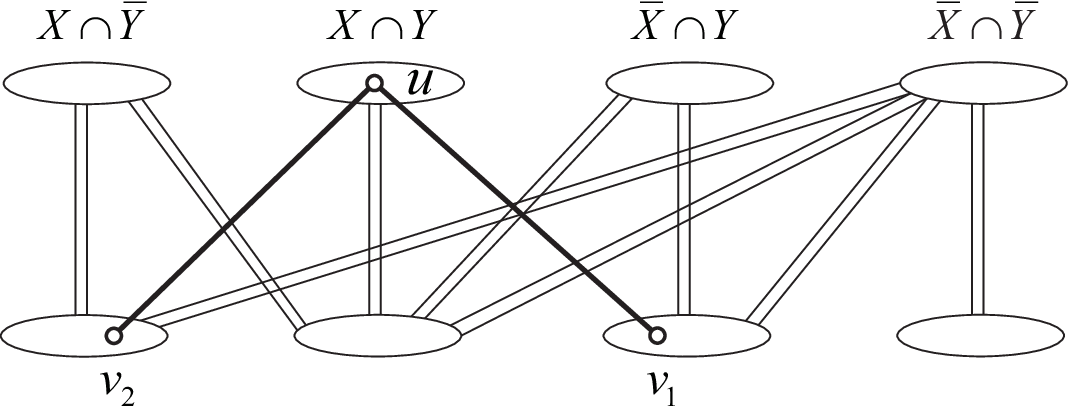}
    \caption{Illustration for the proof of Lemma \ref{lem:nonre-forest}.}
    \label{fig:2-10}
\end{figure}
        \begin{claim}[1]
            $\overline{X}\cap \overline{Y}=\emptyset$.
            \begin{proof}
                Note that $E[X\cap A, \overline{X}\cap B]=\{uv_1\}$, $E[Y\cap A, \overline{Y}\cap B]=\{uv_2\}$, $v_1\in \overline{X}\cap Y$ and $v_2\in \overline{Y}\cap X$.  Then $E[\overline{X}\cap \overline{Y}\cap B, (X\cup Y)\cap A]=\emptyset$. It means that $N(\overline{X}\cap \overline{Y}\cap B)\subseteq \overline{X}\cap \overline{Y}\cap A$.  Then we have $\overline{X}\cap \overline{Y}\cap B=\emptyset$ or $|\overline{X}\cap \overline{Y}\cap B|+1\le|N(\overline{X}\cap \overline{Y}\cap B)|\le |\overline{X}\cap \overline{Y}\cap A|$ by Lemma \ref{lem:bipartite-M-C-Hall}, as $\{v_1,v_2\}\subseteq B\setminus (\overline{X}\cap \overline{Y})$.
                Thus, either $\overline{X}\cap \overline{Y}\cap A=\emptyset$ (which occurs when $\overline{X}\cap \overline{Y}\cap B=\emptyset$), or $|\overline{X}\cap \overline{Y}\cap A|\ge |\overline{X}\cap \overline{Y}\cap B|+1$.

                Suppose to the contrary that $|\overline{X}\cap \overline{Y}\cap A|\ge |\overline{X}\cap \overline{Y}\cap B|+1$.
                Since $|\overline{X}\cap A|=|\overline{X}\cap B|$ (note that $\overline{X}$ is a $P$-set of $G$), we have $|\overline{X}\cap Y\cap A|\le |\overline{X}\cap Y\cap B|-1$. Similarly,  $|X\cap Y\cap A|\ge |X\cap Y\cap B|+1$, as $X$ is a $P$-set.
                Recall that $E[X\cap A, \overline{X}\cap B]=\{uv_1\}$, $E[Y\cap A, \overline{Y}\cap B]=\{uv_2\}$ and $u\in X\cap Y\cap A$. Then $N((X\cap Y\cap A)\setminus\{u\})\subseteq X\cap Y\cap B$.
                By Lemma \ref{lem:bipartite-M-C-Hall} again, we have $(X\cap Y\cap A)\setminus\{u\}=\emptyset$ or $|X\cap Y\cap B|\ge|(X\cap Y\cap A)\setminus\{u\}|+1$.
                It follows that either $X\cap Y\cap B=\emptyset$ or $|X\cap Y\cap B|\ge|(X\cap Y\cap A)\setminus\{u\}|+1$.
                Note that the latter case contradicts $|X\cap Y\cap A|\ge |X\cap Y\cap B|+1$.
                Thus, we have $X\cap Y\cap B=\emptyset$.
                Since $N(u)\setminus\{v_1,v_2\}\subseteq X\cap Y\cap B$ and $u$ is not incident with multiple edges, we have $d_G(u)=2$, contradicting the assumption that $\delta(G)\ge3$.
                Therefore, $|\overline{X}\cap \overline{Y}\cap A|= |\overline{X}\cap \overline{Y}\cap B|$.
                Then $\overline{X}\cap \overline{Y}=\emptyset$, the claim holds.
            \end{proof}
        \end{claim}
        \begin{claim}[2] For any end set $X\cap A$ of $uv_1$ in $A$, there exists an end set $\overline{Y}\cap B$ of $uv_2$ in $B$, such that $|X\cap A|> |\overline{Y}\cap B|$ and $|Y\cap A|> |\overline{X}\cap B|$.
            \begin{proof}
                As $\overline{Y}$ is a $P$-set of $G$, we have $|\overline{Y}\cap A|=|\overline{X}\cap \overline{Y}\cap A|+|X\cap \overline{Y}\cap A|=|\overline{X}\cap \overline{Y}\cap B|+|X\cap \overline{Y}\cap B|=|\overline{Y}\cap B|$.
                Since $\overline{X}\cap \overline{Y}=\emptyset$ (by Claim 1), we have $|X\cap \overline{Y}\cap A|=|X\cap \overline{Y}\cap B|$.
                Then $|X\cap A|=|X\cap Y\cap A|+|X\cap \overline{Y}\cap A|>|X\cap \overline{Y}\cap B|$ (note that $u\in X\cap Y\cap A$).
                Recall that $\overline{X}\cap \overline{Y}=\emptyset$. Then $\overline{Y}\cap B=X\cap \overline{Y}\cap B$. So the claim holds.
            \end{proof}
        \end{claim}

        Suppose to the contrary that $u_1u_2\ldots u_ku_1$  ($k$ is even)  is a cycle of $G$, each edge of which is not removable in $G$, where $\{u_1,u_3,\ldots,u_{k-1}\}\subseteq A$ and $\{u_2,u_4,\ldots,u_k\}\subseteq B$.
        Since $u_1u_2$ and $u_2u_3$  are nonremovable edges of $G$ incident with $u_2$, by Claim 2, there are an end set  $W_1$ of $u_1u_2$ in $B$ and an end set $W_2$ of $u_2u_3$ in $A$ such that $|W_1|>|W_2|$. By Claim 2 again, there exists an end set $W_3$ of $u_3u_4$  in $B$ such that $|W_2|>|W_3|$ as $u_2u_3$ and $u_3u_4$ are incident with $u_3$.
        By using Claim 2 repeatedly, there exist vertex subsets  $W_4$,  $W_5$,  $\ldots$ , $W_{k-1}$, $W_k$ in $G$ such that $|W_4| > |W_5| >\ldots> |W_k|$, where  $W_i$ is an  end set of $u_iu_{i+1}$ in $A$ for each even $4\le i\le k$, and $W_j$ is an end set of $u_ju_{j+1}$ in $B$ for each odd $5\le j\le k-1$ (the subscript is modulo $k$).
        Since $u_1u_2\ldots u_ku_1$ is a cycle, continuing this process we would obtain an infinite sequence of vertex subsets $W_1$, $W_2$, $\ldots$, $W_k$, $W_{k+1}$, $\ldots$ such that $|W_l|>|W_{l+1}|$ for each $l\geq 1$,  which contradicts the fact that $G$ has only finite number of vertices.
        Therefore, the result holds.
    \end{proof}
\end{lem}
\begin{remark}
    {\rm Note that the condition $\delta(G)\ge3$ in Lemma \ref{lem:nonre-forest} is necessary. For example, Lemma \ref{lem:nonre-forest} does not hold for an even cycle with at least 4 vertices, since each edge of it is nonremovable.}
\end{remark}
\begin{cor}\label{cor:nonre-at-most-2}
    Let $G[A,B]$ be a matching covered bipartite graph with at least 4 vertices. If $\delta(G)\ge3$, then at least two vertices in $A$ and at least two vertices in $B$ are incident with at most two nonremovable edges, respectively.
    \begin{proof}
        Suppose to the contrary that there exists at most one vertex of $A$ that is incident with at most two nonremovable edges.
        Then the number of nonremovable edges of $G$ is at least $3(|A|-1)$.
        By Lemma \ref{lem:nonre-forest}, the number of nonremovable edges of $G$ is at most $|V(G)|-1$.
        So $|A|\le 2$. Thus, the underlying simple graph of $G$ is a cycle of order 4.
        It can be checked that every vertex of $G$ is incident with at most one nonremovable edge, as $\delta(G)\ge3$ and every vertex of $G$ has exactly two neighbors, which contradicts the above supposition.
        Therefore, $A$ contains at least two vertices that are  incident with at most two nonremovable edges. So does $B$.
    \end{proof}
\end{cor}

\section{Removable edges in bicritical graphs}

A matching covered  graph that is free of nontrivial tight cuts is  a {\em brick} if it is nonbipartite and a {\em brace}  if it is bipartite.
 Moreover, a graph $G$ is a brick if and only if $G$ is 3-connected and bicritical \cite{ELP82}.
 The {\em tight cut decomposition}, due to Lov\'{a}sz \cite{Lovasz87}, can be applied to a matching covered graph  to produce a unique list of bricks and braces up to multiple edges.

 Lov\'asz \cite{Lovasz87} proved  that every brick different from $K_4$ (the complete graph with 4 vertices) and the triangular prism (the complement of a cycle of length 6) has a removable edge. Improving Lov\'asz's result, Carvalho, Lucchesi and Murty obtained a lower bound of removable classes of a brick in terms of the maximum degree.
\begin{thm}[\cite{CLM99}]\label{thm:re_in_brick}
   Every brick has at least $\Delta(G)$ removable classes. Moreover, every brick different from $K_4$ and the triangular prism has at least $\Delta(G)-2$ removable edges.
\end{thm}

 A matching covered nonbipartite graph $G$ is {\em near-bipartite} if it has a pair of edges $e$ and $f$ such that the subgraph $G-e-f$ is a matching covered bipartite graph.
In fact, every brick with a removable doubleton is near-bipartite \cite{Lovasz87}.
 The following theorem is from an unpublished paper \cite{LV-Unpublished_MS} by Lov\'{a}sz and Vempala; see also Theorem 9.17 of \cite{Lucchesi2024}.
\begin{thm}[\cite{LV-Unpublished_MS}] \label{thm:sim-near-bi-brick}
      Every simple near-bipartite brick distinct from $K_4$, the triangular prism  and $R_8$ (see Figure \ref{fig:r8}) has two nonadjacent removable edges.
\end{thm}

 \begin{figure}[!h]
    \centering
    \includegraphics[totalheight=1.8cm]{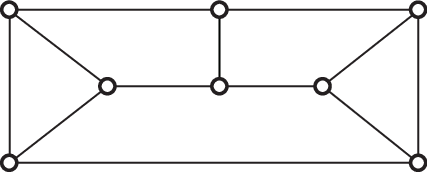}
    \caption{The graph $R_8$.}
    \label{fig:r8}
\end{figure}

We say that a brick $G$ is {\em wheel-like} if there exists a vertex $h$ of $G$, called its {\em hub}, such that for every removable class $R$ of $G$, $|R\cap \partial(h)|=1$.
\begin{lem}[\cite{HLX2025}]\label{lem:WL-brick}
     Let $G$ be a wheel-like brick with at least 6 vertices. Then every edge incident with the hub of $G$ is removable in $G$.
\end{lem}

\begin{lem}\label{lem:brick-cubic-re}
    Let $G$ be a brick with the property that all removable edges of $G$ are incident with the same vertex $h$ and the underlying simple graph of $G$ has a removable doubleton.
    If every 3-line in $G$ (if exists) is incident with $h$, then $|V(G)|=4$, and $G$ has multiple edges between $h$ and each of at least two of its neighbors.
\begin{proof}
    Let $H$ denote the underlying simple graph of $G$.
    Note that every removable edge of $H$ is also removable in $G$.
    So all removable edges of $H$ are incident with $h$.
    Since $H$ is near-bipartite, $H$ is isomorphic to $K_4$, the triangular prism or $R_8$ by Theorem \ref{thm:sim-near-bi-brick}.
    So every edge of $H$ is a 3-line and there exists at least one 3-line of $H$ that is not incident with $h$.
    Thus, $G\neq H$ and hence $G$ has multiple edges.  As each multiple edge is removable, all multiple edges of $G$ are incident with $h$.

    If $H$ is isomorphic to the triangular prism or $R_8$, then there exist $v_1,v_2\in V(G)\setminus\{h\}$, such that $v_1v_2\in E(G)$ and $\{v_1,v_2\}\cap N_G(h)=\emptyset$.
    Recall that every multiple edge of $G$ belongs to $\partial_G(h)$ and every vertex of $H$ is of degree 3.
    Then $d_G(v_1)=3$ and $d_G(v_2)=3$. Thus, $v_1v_2$ is a 3-line of $G$ that is not incident with $h$, which contradicts the assumption.
    So we have $H\cong  K_4$, that is $|V(G)|=4$. As every 3-line of $G$ is incident with $h$, there exist $w_1,w_2\in N_G(h)$ such that $|E[\{h\},\{w_1\}]|\ge2$ and $|E[\{h\},\{w_2\}]|\ge2$. Therefore, the result holds.
  \end{proof}
\end{lem}

\begin{pro}[\cite{LP86}]\label{pro:2-sep-cut-of-bi}
    Let $G$ be a bicritical graph and let $C$ be a 2-separation cut of $G$. Then both $C$-contractions of $G$ are bicritical.
\end{pro}

\begin{pro}[\cite{HLX2025}]\label{pro:degree-2-sep}
    Let $S$ be a 2-separation of a bicritical graph $G$ and let $\partial(X)$ be a 2-separation cut of $G$ that is associated with $S$.
    Then for each $s\in S$, we have $|N(s)\cap (X\setminus S)|\ge2$ and $|N(s)\cap (\overline{X}\setminus S)|\ge2$.
\end{pro}

\begin{pro}\label{pro:bicritical-2-bricks}
     Let $\partial(X)$ be a 2-separation cut of a bicritical graph $G$ associated with the 2-separation $\{u,v\}$ such that $u\in X$. Then the following statements hold. \\
    {\rm (i)} Let $G_1:=G/(\overline{X}\to \overline{x})$. Then $e\in E(G_1)\setminus E_{G_1}[\{u\}, \{\overline{x}\}]$ is removable in $G_1$ if and only if $e$ is removable in $G$. \\
    {\rm (ii)} There exist 2-separation cuts $\partial(Y)$ and $\partial(Y')$ such that $Y'\subseteq\overline{Y}$, and
    both $G/ \overline{Y}$ and $G/ \overline{Y'}$ are bricks.
\begin{proof}
        (i) If $e$ is not incident with $\overline{x}$, then (i) holds by Lemma \ref{lem:re_also_re}.
        Now assume that $e$ is incident with $\overline{x}$.
        Note that $G_2:=G/(X\to x)$ has multiple edges between $v$ and $x$ by Proposition \ref{pro:degree-2-sep}.
        Since $e$ corresponds to a multiple edge between $v$ and $x$ in $G_2$, $e$ is removable in $G_2$. So (i) follows by Lemma \ref{lem:re_also_re} again.

         (ii) Both $\partial(X)$-contractions of $G$ are bicritical by Proposition \ref{pro:2-sep-cut-of-bi}, which have orders less than that of $G$.
        If $G_1$ is 3-connected, it is a brick and let $Y:=X$. Otherwise, $G_1$ has a 2-separation cut by Corollary \ref{cor:2-vtx-cut-of-MC}. We may choose a 2-separation cut $\partial_{G_1}(X_1)$   such that  $\overline x\in \overline X_1=V(G_1)\setminus X_1$  and $ \partial_G(X_1)$  is also a 2-separation cut of $G$.   Since they have a same shore $X_1$, $G_1/(\overline{X_1}\to \overline{x_1})=G/(\overline{X_1}\to \overline{x_1})$ in the sense that $V(G_1)\setminus X_1$ and $V(G)\setminus X_1$ are contracted to the same vertex $\overline{x_1}$. By continuing the above procedure we finally arrive at a 2-separation cut $\partial (Y)$ of $G$ such that $Y\subset X$ and $G/\overline{Y}$ is a brick. For $G_2=G/(X\to x)$, similarly we have a 2-separation cut $\partial(Y')$ of $G$ such that $Y'\subseteq \overline X\subseteq \overline Y$ and $G/\overline{Y'}$ is a brick.
\end{proof}
\end{pro}

 \begin{pro}[\cite{ZWY2022}, Corollary 3.7]\label{pro:WL-bicritical-no-re}
    Let $G$ be a bicritical graph without removable edges. Then $G$ has at least four cubic vertices. As a consequence, every bicritical graph with minimum degree at least 4 has a removable edge.
\end{pro}
 The following property can be implied by the proof of Proposition \ref{pro:WL-bicritical-no-re}. A direct and shorter proof is given here.
\begin{cor} \label{cor:bicritical-wi-re}
    Let $G$ be a bicritical graph without removable edges. Then $G$ has at least two 3-lines and for any edge $f\in E(G)$, $G$ contains a 3-line that  is not adjacent to $f$.
    \begin{proof}
        As $G$ contains no removable edges, $G$ is simple.
        If $G$ is 3-connected, then $G$ is a brick and hence $G$ is isomorphic to $K_4$ or the triangular prism by Theorem \ref{thm:re_in_brick}. So the result holds.
        Now assume that $G$ has a 2-separation by Corollary \ref{cor:2-vtx-cut-of-MC}.
        By Proposition \ref{pro:bicritical-2-bricks} (ii), there exist 2-separation cuts $\partial(X)$ and $\partial(X')$ of $G$ such that $X'\subseteq \overline{X}$, and both $G/(\overline{X}\to\overline{x})$ and $G/(\overline{X'}\to\overline{x'})$ are bricks.
        Let $G_1:=G/(\overline{X}\to\overline{x})$ and let $\{u,v\}$ be a 2-separation of $G$ associated with $\partial(X)$ such that $u\in X$.
        As $G$ is free of removable edges, all removable edges of $G_1$ must be between $\overline{x}$ and $u$ by   Proposition \ref{pro:bicritical-2-bricks} (i).
        So the underlying simple graph of $G_1$ has at most one removable edge.
        Then by Theorem \ref{thm:re_in_brick} and Theorem \ref{thm:sim-near-bi-brick}, the underlying simple graph of $G_1$ is isomorphic to $K_4$, the triangular prism or $R_8$.
        Thus, there exists a 3-line $e_1$ of $G_1$ that is not incident with $\overline{x}$ or $u$. Recalling that $G$ is simple, $e_1$ is a 3-line of $G$.
        Let $\{s,t\}$ be a 2-separation of $G$ associated with $\partial(X')$ such that $s\in X'$ (we allow the possibility of $\{s,t\}=\{u,v\}$).
        Let $G_2:=G/(\overline{X'}\to\overline{x'})$.
        Similarly, $G_2$ contains a 3-line $e_2$ that is not incident with $\overline{x'}$ or $s$, and this 3-line is also a 3-line of $G$.
        Recall that $e_1\in E(G[X\setminus\{u,v\}])$ and $e_2\in E(G[X'\setminus\{s,t\}])\subseteq E(G[\overline{X}\setminus\{u,v\}])$.
        So no edge of $G$ is adjacent to both $e_1$ and $e_2$.
        Thus, the result holds.
    \end{proof}
    \end{cor}

\begin{lem}\label{lem:WL-bicritical}
    Let $G$ be a bicritical graph with a removable edge. If all removable edges of $G$ are incident with a vertex $h$, then $h$ is incident with at least 3 removable edges or there exists a 3-line of $G$ that is not incident with $h$.
   \end{lem}
\begin{proof}
    We proceed by induction on $|V(G)|$.
    If $|V(G)|=4$, then the underlying simple graph of $G$ is isomorphic to $K_4$ and the result holds by Lemma \ref{lem:brick-cubic-re}.
    Suppose that $|V(G)|\ge 6$ and that the result holds for all such graphs with fewer vertices than $G$.
    If $G$ is 3-connected, then $G$ is a brick.
    If the underlying simple graph of $G$ is free of removable doubletons, then $G$ is a wheel-like brick with at least 6 vertices and hence $h$ is incident with at least 3 removable edges by Lemma \ref{lem:WL-brick}.
    Otherwise, the result holds by Lemma \ref{lem:brick-cubic-re}.
    So we may assume that $G$ is not 3-connected.
    As $G$ is bicritical, $G$ has a 2-separation by Corollary \ref{cor:2-vtx-cut-of-MC}.
    By  Proposition \ref{pro:bicritical-2-bricks} (ii), there exists a 2-separation cut $\partial(X)$ of $G$, such that $h\notin X$ and $G/(\overline{X}\to\overline{x})$ is a brick. Let $H_1:=G/(\overline{X}\to \overline{x})$ and $H_2:=G/(X\to x)$. Let $\{u, v\}$ be a 2-separation of $G$ associated with 2-separation cut $\partial(X)$  such that $u \in X$ and $v \in \overline{X}$  (see Figure \ref{fig:4-9}).
    \begin{figure}[h!]
    \centering
    \subfigure{
    \begin{minipage}{5cm}
    \centering
    \includegraphics[totalheight=3.5cm]{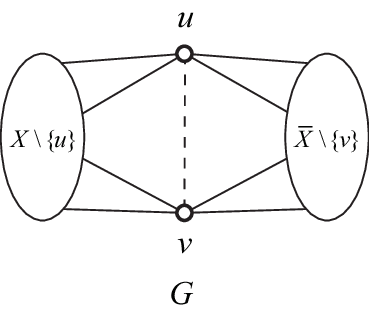}
    \end{minipage}
    }
    \subfigure{
    \begin{minipage}{4.5cm}
    \centering
    \includegraphics[totalheight=3.5cm]{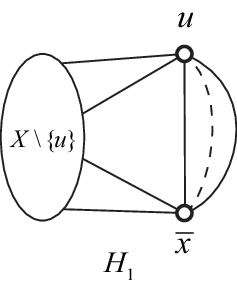}
    \end{minipage}
    }
    \subfigure{
    \begin{minipage}{4.5cm}
    \centering
    \includegraphics[totalheight=3.5cm]{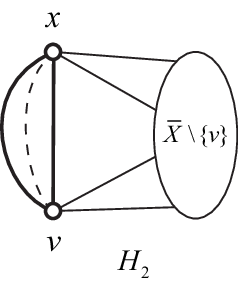}
    \end{minipage}
    }
    \caption{Illustration for the proof of Lemma \ref{lem:WL-bicritical}, where dashed lines indicate the possibility of an edge.}
    \label{fig:4-9}
    \end{figure}

    Since $h\in \overline{X}$, all removable edges of $H_1$ are incident with $\overline{x}$ in $H_1$ by Lemma \ref{lem:re_also_re}.
    If the underlying simple graph of $H_1$ is free of removable doubletons, then $H_1$ has no removable doubletons. So $H_1$ is a wheel-like brick. Since $K_4$, the only simple  brick of order 4, has a removable doubleton, we have $|V(H_1)|\ge6$. By Lemma \ref{lem:WL-brick}, every edge incident with $\overline{x}$ is removable in $H_1$.
    Since $H_1$ is 3-connected, we have $|N_{H_1}(\overline{x})\cap (X\setminus\{u\})|\geq 2$. Let $\{w_1,w_2\}\subset N_{H_1}(\overline{x})\cap (X\setminus\{u\})$.
    As both $\overline{x}w_1$ and $\overline{x}w_2$ are removable in $H_1$, by Proposition \ref{pro:bicritical-2-bricks} (i),  $vw_1$ and $vw_2$ are removable in $G$.
    Thus, $h=v$.
    By Proposition \ref{pro:bicritical-2-bricks} (i) again, every removable edge of $H_2$ is incident with $v$.
    Let $H'_2$ be the graph obtained from $H_2$ by deleting all edges between $x$ and $v$ except exactly two (note that there exist multiple edges between $x$ and $v$ by Proposition \ref{pro:degree-2-sep}).
    Then $H'_2$ is a bicritical graph with a removable edge, such that $4\le|V(H'_2)|<|V(G)|$ and all removable edges are incident with  the  vertex $v$.
    By inductive hypothesis, $H'_2$ contains a 3-line that is not incident with $v$ or $v$ is incident with at least 3 removable edges in $H'_2$.
    For the former case, let $f_1$ be a 3-line of $H'_2$ that is not incident with $v$. By Proposition \ref{pro:degree-2-sep}, $d_{H'_2}(x)\ge4$ and hence $f_1$ is not incident with $x$. Thus, $f_1$ is a 3-line of $G$ that is not incident with $h$ and so, the result holds.
    For the latter case, $H'_2$ contains a removable edge $f_2$ between $v$ and a vertex in $\overline{X}\setminus\{v\}$.
    Obviously $f_2$ is removable in $H_2$. Moreover,  by Lemma \ref{lem:re_also_re},  $f_2$ is removable in $G$.
   Therefore, $f_2$, $vw_1$ and $vw_2$ are three removable edges of $G$ incident with $v$.
   So the lemma holds.

    Now assume that the underlying simple graph of $H_1$ has a removable doubleton.
    The  lemma  holds if $H_1$ contains a 3-line that is not incident with $\overline{x}$.
    So we assume that every 3-line in $H_1$ (if exists) is incident with $\overline{x}$.
    Then by Lemma \ref{lem:brick-cubic-re}, $|V(H_1)|=4$, and $H_1$ has multiple edges between $\overline{x}$ and at least two vertices in $N_{H_1}(\overline{x})$. So $H_1$ contains at least 2 removable edges between $\overline{x}$ and a vertex in $X\setminus\{u\}$.
    By Proposition \ref{pro:bicritical-2-bricks} (i),
    all removable edges of $H_1$ between $\overline{x}$ and a vertex in $X\setminus\{u\}$ are also removable in $G$.
    So $G$ has at least 2 removable edges between $v$ and a vertex in $X\setminus\{u\}$. Recall that $h\in\overline{X}$. Thus, $h=v$.
    The lemma follows by the same argument used in the case when the underlying simple graph of $H_1$ is free of removable doubletons.
   \end{proof}

\section{Proofs of the main results}
For a matching covered nonbipartite graph $G$ with a nontrivial barrier $B$, let $H(G,B)$ denote the graph obtained from $G$ by contracting every nontrivial odd component of $G-B$ to a singleton.
 By Lemma \ref{lem:C-contractions-MC}, $H(G,B)$ is also matching covered. Note that $H(G,B)$ is a bipartite graph with $B$ as one partite set, as $B$ is an independent set and $G-B$ has no even components by Corollary \ref{cor:M-C-without-even-components}.
 Let $I$ be the other partite set of $H(G,B)$.
Let $U_{H(G,B)}:=\{u\in I: u~ \text{is incident with at least 3 edges and at most 2 nonremovable edges in}~ H(G,B)\}$.
When no confusion arises, we assign the same label to any vertex (or edge) common to both $G$ and  $H(G,B)$.
Note that a vertex of $H(G,B)$ but not in $G$ must be gotten by contracting a nontrivial odd component of $G-B$.
\vspace{1em}
\begin{proof}[Proof of Theorem \ref{thm:main-thm}.]
    We proceed by induction on $|V(G)|$.
     By Theorem \ref{thm:2-lines}, $G$ is nonbipartite and simple as $\delta(G)\ge3$ and $G$ is minimal.
    Thus, if $|V(G)|=4$, then $G\cong K_4$ and hence the result holds.
    Now suppose that $|V(G)|\ge 6$ and that the result holds for all such graphs with fewer vertices than $G$.

    If $G$ is free of nontrivial barriers, then $G$ is bicritical by Proposition \ref{pro:mc_is_Bi} and hence, the result holds by Corollary \ref{cor:bicritical-wi-re}.
    From now on, we assume that $G$ has a nontrivial barrier.\\

    \noindent {\bf Case 1.} There exists a nontrivial barrier $B$ of $G$ such that $\delta(H(G,B))=2$.

    Take a vertex $q$ of $H(G,B)$ with degree $2$.
    Since $\delta(G)\geq 3$, $q$ is obtained by contracting a nontrivial odd component $Q$ of $G-B$ to a singleton.
    Let $H:=G/(\overline{V(Q)}\to\overline{q})$. Then $H$ is  a matching covered  and 2-connected graph.
    Since $d_H(\overline{q})=d_{H(G,B)}(q)=2$, every edge incident with $\overline{q}$ is not removable in $H$. Moreover, by Lemma \ref{lem:re_also_re}, $H$ has no removable edges.
    Let $N_{H}(\overline{q}):=\{k_1,k_2\}$ and $X :=\{\overline{q},k_1,k_2\}$.
    Note that $\{k_1,k_2\}$ is a barrier of $H$, as $\overline{q}$ is the vertex of a trivial odd component of $H-\{k_1,k_2\}$. 
    Thus, by Corollary \ref{cor:M-C-without-even-components}, $k_1k_2\notin E(H)$.
    Then for each $i\in\{1,2\}$, we have $|N_H(k_i)\cap \overline{X}|\ge2$, as $H$ is simple and $d_H(k_i)\ge3$ (as $V(H)\setminus\{\overline{q}\}\subseteq V(G)$).
    So $|V(Q)|\ge5$ (note that $Q$ is odd).
    
    Since $|V(H)|\ge6$, $\partial(X)$ is a nontrivial barrier cut of $H$.
    Let $H_1:=H/(\overline{X}\to\overline{x})$ and $H_2:=H/(X\to x)$. Then both $H_1$ and $H_2$ are matching covered.
    Since $H$ is free of removable edges and every edge incident with $\overline{x}$ is removable in $H_1$, by Lemma \ref{lem:re_also_re}, $H_2$ contains no removable edges.
    Note that $V(H_2)\setminus\{x\}\subseteq V(G)$.
    Thus, $H_2$ is a minimal matching covered graph with $4\le |V(H_2)|<|V(G)|$ and $\delta(H_2)\ge3$.
    By inductive hypothesis, there exist two nonadjacent 3-lines, say $e_1$ and $e_2$ of  $H_2$.
    Since $|\partial(X)|\ge4$ (as $|N_G(k_i)\cap V(Q)|\ge2$, for each $i\in\{1,2\}$), we have $d_{H_2}(x)\ge4$.
    Then $e_1$ and $e_2$ are nonadjacent 3-lines of $H$ and hence also in $G$ (note that $d_H(\overline{q})=2$). So the result holds for this case. \\

    \noindent {\bf Case 2.} For every nontrivial barrier  $B$  of $G$, $\delta(H(G,B))\ge3$.

    We have the following claims for this case.
    \begin{claim}[1]
        For every nontrivial barrier $B$ of $G$, we have $|U_{H(G,B)}|\ge2$ and $V(G)\cap U_{H(G,B)}=\emptyset$.
        \begin{proof}
            Since $\delta(H(G,B))\ge3$, by Corollary \ref{cor:nonre-at-most-2}, we have $|U_{H(G,B)}|\ge2$.
            Suppose, to the contrary, that there exists a vertex $u\in V(G)\cap U_{H(G,B)}$.
            Since every vertex of $U_{H(G,B)}$ is incident with some removable edge of  $H(G,B)$ (as $\delta(H(G,B))\ge3$), we may assume that $uv$ is removable in $H(G,B)$, where $v\in B$.
            Then by Lemma \ref{lem:re_also_re}, $uv$ is a removable edge of $G$, as $u,v\in V(G)$.
            It contradicts the assumption that $G$ has no removable edges.
            So the claim holds.
        \end{proof}
    \end{claim}

    Let $B_0$ be a maximal nontrivial barrier of $G$.
By Claim 1, $|U_{H(G,B_0)}|\geq 2$. So we take two distinct vertices  $q_0$ and $q_0'$ in $ U_{H(G,B_0)}$, which are obtained by contracting
      nontrivial odd components $Q_0$ and $Q_0'$ of $G-B_0$ respectively.
    Let $G_1:=G/(\overline{V(Q_0)}\to\overline{q_0})$ and $G_1':=G/(\overline{V(Q_0')}\to\overline{q_0'})$.
By Lemma \ref{lem:max-barrier-critical}, $Q_0$ and $Q_0'$ are nonbipartite, so $G_1$ and $G'_1$ are  nonbipartite matching covered graphs by Lemma \ref{lem:C-contractions-MC}. Next it suffices to show that $G_1$ (resp.  $G'_1$) has a 3-line not incident with $\overline{q_0}$ (resp. $\overline{q_0'}$) (obviously, the 3-line in $G_1$ and the one in $G'_1$ are disjoint 3-lines in $G$). We only consider $G_1$, the case of $G'_1$ is the same.

    \begin{claim}[2] If $G_1$ is not bicritical, let $B_1$ be a maximal nontrivial barrier of $G_1$.
        Then  $\overline{q_0}\notin B_1$, $B_1$ is also a barrier of $G$ disjoint $B_0$ and $H(G_1,B_1)=H(G,B_1)$.
        \begin{proof}
           If $\overline{q_0}\in B_1$, then $B_0\cup (B_1\setminus\{\overline{q_0}\})$ is a barrier of $G$.
            But $B_0\subset  B_0\cup (B_1\setminus\{\overline{q_0}\})$, which contradicts the maximality of $B_0$.  Hence $\overline{q_0}\notin B_1$.
            Further  $\overline{q_0}$ is contained in an odd component of $G_1-B_1$.
            It can be checked that $B_1$ is a nontrivial barrier of $G$ and $H(G,B_1)=H(G_1,B_1)$.
            \end{proof}
            \end{claim}
            Let $Q'_1$ be the component of $G_1-B_1$  such that $\overline{q_0}\in V(Q'_1)$ (we allow the possibility of $V(Q'_1)=\{\overline{q_0}\}$).
            Let $q'_1$ be the vertex gotten by contracting $Q'_1$ to a singleton.  By Claim 2, $U_{H(G_1,B_1)}=U_{H(G,B_1)}$. Hence $|U_{H(G_1,B_1)}|\ge2$  by Claim 1.
            Then there exists a vertex $q_1\in U_{H(G_1,B_1)}$ such that $q_1\neq q'_1$.
            By Claim 1 again, $q_1$ is obtained by contracting a nontrivial odd component $Q_1$ of $G_1-B_1$, as $U_{H(G_1,B_1)}=U_{H(G,B_1)}$.
            Since $V(G_1)\setminus\{\overline{q_0}\}\subseteq V(G)$, we have $V(Q_1)\subset V(G)$ (see Figure \ref{fig:1-1}).

    \begin{figure}[h!]
    \centering
    \subfigure{
    \begin{minipage}{6cm}
    \centering
    \includegraphics[totalheight=3.5cm]{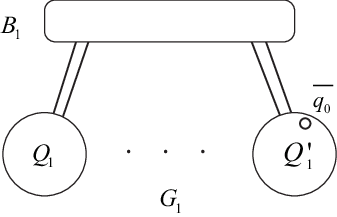}
    \end{minipage}
    }
    \subfigure{
    \begin{minipage}{6cm}
    \centering
    \includegraphics[totalheight=3.5cm]{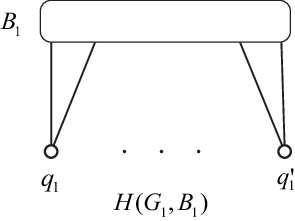}
    \end{minipage}
    }
    \subfigure{
    \begin{minipage}{2.5cm}
    \centering
    \includegraphics[totalheight=3.5cm]{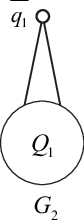}
    \end{minipage}
    }
    \caption{Illustration for the proof of Theorem \ref{thm:main-thm}.}
    \label{fig:1-1}
    \end{figure}
    Let $G_2:=G_1/(\overline{V(Q_1)}\to\overline{q_1})$. Then $G_2$ is also a nonbipartite matching covered graph.
    If $G_2$ is not bicritical,  let $B_2$ be a maximal nontrivial barrier of $G_2$. Similar to  Claim 2,   by the maximality of $B_1$ in $G_1$ we have that $\overline{q_1}\notin B_2$, $B_2$ is also a nontrivial barrier of $G$ and $H(G,B_2)=H(G_1,B_2)=H(G_2,B_2)$. Similarly let $Q_2$ and $Q_2'$ be  distinct odd components of $G_2-B_2$ contracted to singletons $q_2$ and $q_2'$ in $H(G_2,B_2)$ such that   $ q_2\in U_{H(G_2,B_2)}$ and $\overline q_1\in V(Q_2')$. Further $Q_2\subset Q_1\subset G$.

    For this time let $G_3:=G_2/(\overline{V(Q_2)}\to\overline{q_2})$.
    By continuing above procedure, since $G$ (write $G_0$)  is finite we finally obtain a nonbipartite matching covered graph $G_s$, $s\geq 0$,  satisfying the following: \\
    \indent  (i) For $s\ge 1$, $V(G_s)\setminus\{\overline{q_{s-1}}\}\subset V(G)$ and $G_s$ has a maximal nontrivial barrier $B_s$ such that $\overline{q_{s-1}}\notin B_s$; and\\
     \indent (ii) there exists a nontrivial odd component $Q_s$ of $G_s-B_s$, such that $G_{s+1}:=G_s/(\overline{V(Q_s)}\to\overline{q_s})$ has no nontrivial barriers, $V(Q_s)\subset V(G)$ and $q_s\in U_{H(G_s,B_s)}=U_{H(G,B_s)}$, where $q_s$ is obtained by contracting the odd component $Q_s$ of $G_s-B_s$.

Since  $G':=G_{s+1}$ has no nontrivial barriers, by Proposition \ref{pro:mc_is_Bi}, $G'$ is bicritical.
    If $G'$ is free of removable edges, then by Corollary \ref{cor:bicritical-wi-re}, $G'$ contains at least two nonadjacent 3-lines  and hence there exists a 3-line $f_1$ of $G'$ that is not incident with $\overline{q_s}$.
    Since $V(G')\setminus\{\overline{q_s}\}\subseteq V(Q_0)\subseteq V(G)$, $f_1$ is a 3-line of $G$ that lies in  $Q_0$.
    Now assume that $G'$ contains a removable edge.
    Then by Lemma \ref{lem:re_also_re}, all removable edges of $G'$ are incident with $\overline{q_s}$, as $G$ is free of removable edges and $V(G')\setminus\{\overline{q_s}\}\subseteq V(G)$.
    Recall that $q_s\in U_{H(G_s,B_s)}$, that is, $q_s$ is incident with at most two nonremovable edges and at least one removable edge in $H(G_s,B_s)$.
    Then by Lemma \ref{lem:re_also_re}, $q_s$ is incident with at most two nonremovable edges and at least one removable edge in $G_s/(V(Q_s)\to q_s)$ as well.
    Hence, by Lemma \ref{lem:re_also_re} again, oppositely  $\overline{q_s}$ is incident with at most two removable edges of $G'$ and at least one nonremovable edge.
    Thus, by Lemma \ref{lem:WL-bicritical}, there exists a 3-line  of $G'$ that is not incident with $\overline{q_s}$, which lies in $Q_0$.
\end{proof}

By some slight changes to the proof of Theorem \ref{thm:main-thm}, we can derive the following result.
\begin{cor}\label{cor:M-C-only-1-RE}
    Let $G$ be a matching covered graph with $\delta(G)\ge3$ and at most one removable edge, and let $e$ denote  the only removable edge of $G$ if it exists; otherwise, let $e$ be any edge of $G$. Then for each end vertex $z$ of $e$, $G$ contains a 3-line that is not incident with $z$. 
\end{cor}
\begin{proof}
    We proceed by induction on $|V(G)|$.
    By Corollary \ref{cor:nonre-at-most-2}, $G$ is nonbipartite.
    For $|V(G)|=4$, $K_4$ is a unique such graph.  Obviously, it has a 3-line and no removable edges, and thus the result holds.
    So we assume that $|V(G)|\ge6$ and that the result holds for all such graphs with fewer  vertices than $G$.
    If $G$ has no removable edges, then by Theorem \ref{thm:main-thm} it contains two nonadjacent 3-lines and so, for each end vertex $z$ of $e$, there exists a 3-line not incident with $z$. 
    So we assume that $G$ has just one removable edge.
    If $G$ is free of nontrivial barriers, then  $G$ is bicritical by Proposition \ref{pro:mc_is_Bi}, and hence  the result holds by Lemma \ref{lem:WL-bicritical}.
    So we assume that $G$ has a nontrivial barrier.

    If there exists a nontrivial barrier $B$ of $G$, such that $\delta(H(G, B)) = 2$, then by the analogous argument as  Case 1 of the proof of Theorem \ref{thm:main-thm}, we can show that $G$ contains a 3-line that is not incident with $z$. 
    Now assume that for every nontrivial barrier $B'$ of $G$, $\delta(H(G, B'))\ge 3$.
    Let $B_0$ be a maximal nontrivial barrier of $G$. By Corollary \ref{cor:nonre-at-most-2}, we have $|U_{H(G,B_0)}|\ge2$.
    If there exists a vertex $w\in U_{H(G,B_0)}\cap V(G)$, then $w$ is incident with a removable edge $e$ of $H(G,B_0)$ and so, by Lemma \ref{lem:re_also_re}, $e$ is removable in $G$.
    Since $G$ has exactly one removable edge, we have $\{w\}=U_{H(G,B_0)}\cap V(G)$.
     Then for each $q\in U_{H(G,B_0)}\setminus\{w\}$, $q$ is obtained by contracting a nontrivial odd component $Q$  of $G-B_0$. Moreover, every vertex of $Q$ is not incident with a removable edge of $G$.
    If $U_{H(G,B_0)}\cap V(G)=\emptyset$, then every vertex of $U_{H(G,B_0)}$ is obtained by contracting a nontrivial odd component of $G-B_0$.
    So, whether $U_{H(G,B_0)}\cap V(G)$ is $\emptyset$ or not,  we can choose a nontrivial odd component $Q_0$ of $G-B_0$, such that $q_0\in U_{H(G,B_0)}$ and no vertex of $Q_0$ is incident with a removable edge of $G$,  where $q_0$ is the vertex obtained by contracting $Q_0$.
    By a similar argument as Case 2 of the proof of Theorem \ref{thm:main-thm}, we can show that $G$ contains a 3-line with both end vertices lying in $Q_0$.
\end{proof}

\begin{proof}[Proof of Theorem \ref{thm:2-or-3-lines}.]
    We proceed by induction on $|V(G)|$.
    If $|V(G)|=4$, then $G$ is isomorphic to $K_4$ or a cycle with 4 vertices. Thus, the result holds.
    Now suppose that $|V(G)|\ge 6$ and that the result holds for all such graphs with fewer vertices than $G$.
    If $\delta(G)\ge3$, then by Theorem \ref{thm:main-thm}, the result holds.
    So we consider the case when $\delta(G)=2$.
    Assume first that $G$ has a vertex $u$ of degree 2 such that $d(u_1)=d(u_2)=2$, or $d(u_1)\geq 3$ and $d(u_2)\geq 3$, where $\{u_1,u_2\}=N(u)$.
    Note that $\{u_1,u_2\}$ is a barrier of $G$ and  $u_1u_2\notin E(G)$ by Corollary \ref{cor:M-C-without-even-components}.
    Let $Y:=\{u,u_1,u_2\}$. Since $|V(G)|\ge6$, we have $|\overline{Y}|\ge3$ and hence $\partial_G(Y)$ is a nontrivial barrier cut of $G$.
    Let $G_1:=G/(Y\to y)$ and $G_2:=G/(\overline{Y}\to \overline{y})$.

    If $d_G(u_1)=d_G(u_2)=2$, then $|\partial_G(Y)|=2$ and hence $d_{G_1}(y)=2$.
    Thus, every edge incident with $y$ is not removable in $G_1$.
    By Lemma \ref{lem:re_also_re}, $G_1$ is free of removable edges, as $G$ is minimal.
    If $d_G(u_1)\ge3$ and $d_G(u_2)\ge3$, then $|E[\{u_1\},\{\overline{y}\}]|\ge2$ and $|E[\{u_2\},\{\overline{y}\}]|\ge2$.
    Thus, every edge incident with $\overline{y}$ is removable in $G_2$ and then, by Lemma  \ref{lem:re_also_re}, every edge incident with $y$ is not removable in $G_1$, as $G$ is minimal.
    Therefore, in both cases, $G_1$ is a minimal matching covered graph.

    Since $4\le|V(G_1)|<|V(G)|$, by inductive hypothesis, $G_1$ contains two nonadjacent edges, say $e_1$ and $e_2$, such that each of $e_1$ and $e_2$ is a 2-line or 3-line.
    Then at least one of $e_1$ and $e_2$, say $e_1$, is not incident with $y$.
    So $e_1$ is also a 2-line or 3-line in $G$.
    If $d_G(u_1)=d_G(u_2)=2$, then  $uu_1$  is a 2-line of $G$ that is not adjacent to $e_1$ and hence, the result holds.
     If $d_G(u_1)\ge3$ and $d_G(u_2)\ge3$, then $d_{G_1}(y)=d_{G_2}(\overline{y})\ge4$ (recall that $|E[\{u_1\},\{\overline{y}\}]|\ge2$ and $|E[\{u_2\},\{\overline{y}\}]|\ge2$).
    So neither $e_1$ nor $e_2$ is incident with $y$.
    Thus, $e_1$ and $e_2$ are two nonadjacent edges of $G$.
    Therefore, the result holds.

    Now we assume that for every vertex of degree 2 in $G$, one of its neighbors is of degree two, and the other neighbor is of degree at least three.
    Let $v\in V(G)$, $d_G(v)=d_G(v_1)=2$ and $d_G(v_2)\ge3$, where $N_G(v)=\{v_1,v_2\}$.
    Let $Z:=\{v,v_1,v_2\}$. Similar to the above, we have that $\partial(Z)$ is a nontrivial barrier cut of $G$, $\{v_1,v_2\}$ is a barrier of $G$ and  $v_1v_2\notin E(G)$.
    Thus, $|E[\{v_1\},\overline{Z}]|=1$ and $|E[\{v_2\},\overline{Z}]|\ge2$.
    Let $H_1:=G/(Z\to z)$ and $H_2:=G/(\overline{Z}\to \overline{z})$.

    Note that in $H_2$, $\overline{z}v_1$ is the only nonremovable edge incident with $\overline{z}$ since $d(v_1)=2$.
    By Lemma \ref{lem:re_also_re}, every removable edge of $H_1$ is incident with $z$ and $z$ is incident with at most one removable edge.
    Assume first that $\delta(H_1)\ge3$. If $H_1$ is free of removable edges, then $H_1$ contains two nonadjacent 3-lines by Theorem \ref{thm:main-thm}. If $H_1$ has exactly one removable edge, then $H_1$ contains a 3-line that is not incident with $z$ by Corollary \ref{cor:M-C-only-1-RE}. Thus, whether $H_1$ has removable edges or not, there exists a 3-line $f_1$ in $H_1$ that is not incident with $z$.
    So $f_1$ and $vv_1$ are the two required nonadjacent edges.
    Next we consider the case  that $H_1$ has  a vertex $w$ of degree 2.
    Since $d_{H_1}(z)\ge3$, we have $w\neq z$. So $w\in V(G)$.
    By the above assumption,  every vertex of degree 2 in $G$ has exactly one neighbor of degree 2, so $w$ is not  a neighbor of $v_1$. Then there exists a vertex $w_1\in \overline{Z}$ such that $w_1\in N_G(w)$ and $d_G(w_1)=2$ (note that $N_G(w)\subseteq \overline{Z}\cup\{v_2\}$).
    Thus, $ww_1$ is a 2-line that is nonadjacent to 2-line $vv_1$.
    Therefore, the theorem follows.
\end{proof}  
\section{A conjecture}A connected graph $G$ is {\em $n$-extendable} if it  contains at  least $2n+2$ vertices and a matching of size at least $n$,  and every matching of size $n$ is a subset of a perfect matching of $G$. So
a matching covered graph with at least four vertices is a 1-extendable graph, and a brace with at least six vertices is a 2-extendable bipartite graph \cite{Lovasz87}. For a minimal brace $G$ on at least 6 vertices, Lou \cite{Lou1999} proved that the minimum degree of  $G$ is 3 and  $G$ has at least $\lceil\frac{2|V(G)|+2}{5} \rceil$ cubic vertices.
We surmise the following:
\begin{conj}\label{conj:1}
Every minimal brace with at least 6 vertices has a 3-line.
\end{conj}
A proof of the above conjecture has already been announced by Zhang and Wang \cite{ZW2026}.


\begin{thebibliography}{99}

\parskip -0.08cm

\bibitem{BM08} J. A. Bondy, U. S. R. Murty, Graph Theory, Springer-Verlag, Berlin, 2008.
 
\bibitem{CLM99} M. H. Carvalho, C. L. Lucchesi, U. S. R. Murty, Ear decompositions of matching covered graphs, Combinatorica  19 (2) (1999) 151--174.
 

 

 \bibitem{CLM15} M. H. Carvalho, C. L. Lucchesi, U. S. R. Murty, Thin edges in braces, Electron. J. Combin.  22 (4) (2015) \#P4.14. 


\bibitem{ELP82} J. Edmonds, L. Lov\'asz, W. R. Pulleyblank, Brick decompositions and the matching rank of graphs, Combinatorica  2 (3) (1982) 247--274. 

\bibitem{HLX2025}
X. He, F. Lu, J. Xue, Wheel-like bricks and minimal matching covered graphs, J. Graph Theory  111 (2026) 5--16. 

\bibitem{Lou1999} D. Lou, On the structure of minimally $n$-extendable bipartite graphs, Discrete Mathematics 202 (1999) 173--181.  

\bibitem{LP77}
L. Lov\'{a}sz, M. D. Plummer, On minimal elementary bipartite graphs,
J. Combin. Theory, Ser. B  23 (1977) 127--138.


\bibitem{LP86} L. Lov\'{a}sz, M. D. Plummer, Matching Theory, Annals of Discrete Mathematics, vol. 29, Elsevier Science, 1986.

\bibitem{Lovasz87} L. Lov\'{a}sz, Matching structure and the matching lattice, J. Combin. Theory, Ser. B  43 (1987) 187--222. 

\bibitem{LV-Unpublished_MS}  
L. Lov\'{a}sz, S. Vempala,
On removable edges in matching covered graphs,
unpublished manuscript.
 

\bibitem{Lucchesi2024}
C. L. Lucchesi, U. S. R. Murty, Perfect Matchings, Springer, 2024.
 

\bibitem{Tutte47} W. T. Tutte, The factorization of linear graphs, J. Lond. Math. Soc. 22 (1947) 107--111.

\bibitem{ZWY2022} Y. Zhang, X. Wang, J. Yuan, Bicritical graphs without removable edges, Discrete Applied Mathematics 320 (2022) 1--10. 

\bibitem{ZW2026} Y. Zhang, X. Wang, Cubic vertices in minimal braces, arXiv:2610.06933v1 (2026). 
\end{thebibliography}
\end{document}